\documentclass[12pt,a4paper]{amsart}
\usepackage{amsfonts}

\def\N{{\mathbb N}}

\def\f{{\mathbb F_p}}

\usepackage{amssymb}
\newtheorem{thm}{\bf Theorem}[section]
\newtheorem{lemma}[thm]{\bf Lemma}
\newtheorem{prop}[thm]{\bf Proposition}

\theoremstyle{definition}

\begin{document}

 \title[A structure result for bricks in Heisenberg groups]{A structure result for bricks in Heisenberg groups}

 \author[]{Norbert Hegyv\'ari}
 \address{Norbert Hegyv\'{a}ri, ELTE TTK,
E\"otv\"os University, Institute of Mathematics, H-1117
P\'{a}zm\'{a}ny st. 1/c, Budapest, Hungary}
 \email{hegyvari@elte.hu}

 \author[]{Fran\c cois Hennecart}
 \address{Fran\c cois Hennecart, Universit\'e Jean-Monnet,
Institut Camille Jordan,
   23, rue du Docteur Paul Michelon, 42023 Saint-Etienne Cedex 02, 
   France}
 \email{francois.hennecart@univ-st-etienne.fr}

\date{\today}

\begin{abstract}
We show that for a sufficiently big \textit{brick} $B$ of the $(2n+1)$-dimensional Heisenberg group $H_n$ over the finite field
$\mathbb{F}_p$, the product set $B\cdot B$ contains at least $|B|/p$ many cosets of some non trivial subgroup 
of $H_n$.
% and $|S|>\big ({3\over 2}\big )^m{|X|\over |G|}$, for some
%$m\ge n+4-\log|G|/\log p$.
\end{abstract}



%\thanks{The second author is grateful to the members
%of E\"otv\"os University and R\'enyi Institute for their warm hospitality during his stay}
 \maketitle

 \section{\bf Introduction}

The notion of additive bases occupies
a central position in Combinatorial Number Theory. In an additive semigroup $G$, a \textit{basis} means a subset $A$ of $G$ such that there exists an integer $h$, depending only on $A$, for which any element of $x\in G$ can be written as a sum of $h$
(or at most $h$)  members of $A$. The idea has been widely investigated in different structures in which a number of results have been shown. One can quote the celebrated Lagrange theorem in the set of nonnegative integers but also such results in $\sigma$-finite abelian groups \cite{r1}. 

 In an additive structure we will use the notation $A+B=\{a+b\,:\, a\in A,\ b\in B\}$, its extension $hA=A+A+\cdots+A$ ($h$ times) and also their counterparts $A\cdot B$, $A^h$ in a multiplicative structure. In a group we also denote $-A$ (resp. $A^{-1}$) for the set of the inverses of elements of $A$.
With this notation, $A$ is a basis in $G$ whenever for some integer $h$ one has $hA=G$ or $A^h=G$ according to the underlying structure. One also defines the notion of \textit{doubling constant}
(resp. \textit{squaring constant})
of a finite set $A$ that is $|A+A|/|A|$ (resp.  $|A\cdot A|/|A|$).

Another aspect concerns inverse results in number theory in which the Freiman theorem has a central place. It asserts that a finite set  $A$ with a small doubling constant in an abelian (additive) group $G$ has a sharp structure, namely it is  included, as a rather dense subset, in a (generalized) arithmetic progression of cosets of some subgroup of $G$ (cf. \cite{GR}). 
An important tool for the proof, known as the Bogolyubov-Ruzsa Lemma, is the fact that $2A-2A$ contains a dense substructure.

According to  the preceding discussion, one may consider the general problem of investigating in which conditions 
on a finite $A$, the sumset $A+A$ (or the product set $A\cdot A$) contains a \textit{rich} substructure.
In this paper we will focus on Heisenberg groups which
 give an interesting counterpoint  to the commutative case.

Let $p$ be a prime number and $\mathbb{F}$ 
the field with $p$ elements.
We denote by $H_n$ the $(2n+1)$-dimensional Heisenberg linear group over $\mathbb{F}$ formed with the 
upper triangular square matrices of size $n+2$ of the 
following kind
$$
[\underline{x},\underline{y},z]=\begin{pmatrix}
1 & \underline{x} & z\\
0 & I_n & ^{t}\underline{y}\\
0 & 0 & 1
\end{pmatrix},
$$
where $\underline{x}=(x_1,x_2,\dots ,x_n)$, $\underline{y}=(y_1,y_2,\dots ,y_n)$, $x_i,y_i,z\in\mathbb{F}$, $i=1,2,\dots, n$,
and $I_n$ is the $n \times n$ identity matrix.
We have $|H_n|=p^{2n+1}$. and we recall the product rule in $H_n$:
$$
[\underline{x},\underline{y},z][\underline{x'},\underline{y'},z']=
[\underline{x}+\underline{x'},\underline{y}+\underline{y'},\langle\underline{x},\underline{y'}\rangle+z+z'],
$$
where $\langle\cdot, \cdot \rangle$ is the inner product, that is $\langle\underline{x},\underline{y}\rangle=\sum_{i=1}^nx_iy_i.$

So this set of $(n+2)\times (n+2)$ matrices form a group whose unit is $e=[\underline{0},\underline{0},0]$. 
As group-theoretical properties of $H_n$, we recall that
$H_n$ is non abelian and two-step nilpotent, that is
the double commutator satisfies $aba^{-1}b^{-1}cbab^{-1}a^{-1}c^{-1}=e$ for any
$a,b,c\in H_n$, where the commutator of $a$ and $b$ is defined as $aba^{-1}b^{-1}$.

The Heisenberg group possesses an interesting structure in which we can prove that in general there is no \textit{good model} for a subset $A$ with a small \textit{squaring constant} $|A\cdot A|/|A|$ (see \cite{r2}, \cite{r3} for more details), unlike for subsets of abelian groups. We should add that the situation is less unusual if we assume that $A$ has a small \textit{cubing constant} $|A\cdot A\cdot A|/|A|$ (see \cite{T2}).

%Here by a good $s$-model for $A$ we mean a subset $A'$ of a finite group $G'$ such that
%$A$ is $s$-Freiman isomorphic to $A'$ and $|G'|\ll_{s,K}|A'|$ where
%$K=|A\cdot A|/|A|$ is the squaring ratio (see \cite{r3} for more details).

We now quote the  following well-known results.
\begin{lemma}\label{l1} Let $X$ and $Y$ be subsets of a finite (multiplicative) group $G$. 
If $|X|+|Y|>|G|$ then $G=X\cdot Y$.
\end{lemma}
The proof follows from the simplest case of the sieve formula:
 $$
|X\cap(\{x\}\cdot Y^{-1})|=|X|+|\{x\}\cdot Y^{-1}|-|X\cup(\{x\}\cdot Y^{-1})|\ge|X|+|Y|-|G|>0.
$$
\begin{lemma}\label{l2} Let $X$ and $Y$ be subsets of $\mathbb{F}$. 
If $X+Y\ne \mathbb{F}$ then $|X+Y|\ge|X|+|Y|-1$.
\end{lemma}
This general  lower bound for the cardinality of sumsets in $\mathbb{F}$ is known as the Cauchy-Davenport Theorem (see e.g. \cite{TV}).

We deduce from Lemma \ref{l1} that a sufficient condition ensuring
that a subset $A\subseteq H_n$ is a basis is $|A|>|H_n|/2$. Moreover this condition is sharp if $p=2$ since in that case $H_n$ has a subgroup of index $2$. For $p>2$, any subset of $H_n$
with cardinality bigger than $|H_n|/p$ is not contained 
in a coset of a proper subgroup of $H_n$, hence it
is a basis for some order $h$ bounded by a function  depending
only on $p$: indeed by a theorem of Freiman in arbitrary finite groups
(see \cite{Tao}, paragraph 4.9), 
it is known that if $A$ is not included in some coset of some
proper subgroup of $H_n$ then $|A\cdot A|\ge3|A|/2$. From this 
we  deduce by iteration that the $2^j$-fold product set
$A^{2^j}$ satisfies $|A^{2^j}|>|H_n|/2$ for
$j\ge \ln(p/2)/\ln(3/2)$, hence the result by Lemma~\ref{l1}.

The above discussion shows that any sufficiently dense subset of $H_n$ is a basis. This does not hold true in general for sparse subsets. Another main consideration is that the \textit{squaring constant} $|A\cdot A|/|A|$ of a $A\subset H_n$ is not necessarily  big. 
%A general subset $A$ is not necessarily a basis and even has not necessarily a big squaring constant. 
So we can ask the following: is it true that for any subset $A$ of $H_n$ which is large enough, the product set $A\cdot A$ always contains some non-trivial substructure of $H_n$ ?
A dual question emerged in $\mathrm{SL}_2(\mathbb{F})$ in \cite{r6} (see also \cite{r5}) where it is proved that for any generating  subset $A$ of $\mathrm{SL}_2(\mathbb{F})$ such that $|A|<p^{3-\delta}$, one has $|A\cdot A\cdot A|>|A|^{1+\varepsilon}$ with $\epsilon$ depending only on  $\delta>0$.

Nevertheless, the structure cannot be handled in general.
We will restrict our attention to subsets that will be called \textit{bricks}.

Let $B\subseteq H_n,$ and write the projections of $B$ onto each coordinates by $X_1,X_2,\dots ,X_n$, $Y_1,Y_2,\dots ,Y_n$ and $Z$, i.e. one has 
$[\underline{x}, \underline{y},z]\in B$, 
 $\underline{x}=(x_1,x_2,\dots ,x_n), \underline{y}=(y_1,y_2,\dots ,y_n)$, if and only
 if $x_i\in X_i$ or $y_i\in Y_i$ for some $i$, or $z\in Z$. 
 
A subset $B$ of $H_n$ is said to be a \textit{brick} if
$$
B=[\underline{X},\underline{Y},Z]:=\{[\underline{x}, \underline{y},z]\text{ such that }\underline{x}\in\underline{X},\   \underline{y}\in\underline{Y},\ z\in Z\}
$$
where $\underline{X}=X_1\times\cdots\times X_n$ and
$\underline{Y}=Y_1\times\cdots\times Y_n$ with 
non empty-subsets $X_i,Y_i\subset\mathbb{F}^*$.
If $|Z|>p/2$ then $2Z=\mathbb{F}$  by Lemma \ref{l1}, hence
$B\cdot B=[2\underline{X},2\underline{Y},\mathbb{F}]$. This shows that $B\cdot B$ is a union of $|B\cdot B|/p\ge |B|/p$ cosets of
the subgroup $[\underline{0},\underline{0},\mathbb{F}]$ of $H_n$.
Our aim is to partially extend this result under an appropriate assumption on the size of $B$.

\begin{thm}\label{th1}
For every $\varepsilon>0,$ there exists a positive integer $n_0$ such that if $n\ge n_0$, $B\subseteq H_n$ is a brick and
$$
|B|>|H_n|^{3/4+\varepsilon}
$$ 
then there exists a non trivial subgroup $G$ of $H_n$, namely its center $[\underline{0},\underline{0},\mathbb{F}]$, such that $B\cdot B$ contains  a union of %$M:=\big(\frac32\big)^m\frac{|B|}{|G|}$ 
at least $|B|/p$ many cosets of   $G$. 
%where $m\geq n+4-\log |G|/ \log p$. % and provided $p>2^{n(n-2)/2}.$ 
%On other words, we have 
%$$
%B\cdot B=S\cdot G,\quad \text{where $|S|\geq M$.}
%$$
\end{thm}
We stress the fact that $n_0$ depends only on $\varepsilon$
and that this result is valid uniformly in $p$.

In the case when $B$ has a small squaring constant, namely 
$|B\cdot B|<3 |B|/2$, we already observed that $B\cdot B$ is a coset modulo some subgroup of $H_n$ (see \cite{Tao}).
Theorem \ref{th1} thus provides a partial extension of this fact by giving a substructure result  for 
sufficiently big bricks $B$ in $H_n$.  

The  statement in Theorem \ref{th1} can be plainly extended  to any subset $B'\subset H_n$  which derives from a brick $B$ by conjugation : $B'=P^{-1}BP$ where $P$ is a given element of $H_n$.

As a remark also in connection with  the above-mentioned Helfgott's result, we notice that for a brick $B$ of the Heisenberg group $H_n$  its \textit{squaring set}
$B\cdot B$ could have a small period, namely the maximal subgoup $P$ such that $B\cdot B\cdot P=B\cdot B$ and at the same time
its squaring constant $|B\cdot B|/|B|$ could be small. To see this
take the brick $B=\{[x,y,z],\ 0\le x, y< \sqrt{p}/2,\ z\in\mathbb{F}\}$
in $H_1$ (the general case $n\ge1$ could be easily derived from 
this discussion).
Then $B$ generates the full group $H_1$ and its cardinality 
satisfies $|B|\asymp p^2$. Furthermore
$B\cdot B=\{[x,y,z],\ 0\le x,y< \sqrt{p},\ z\in\mathbb{F}\}$
has cardinality less than $4|B|$ and its period is the 
subgroup $[{0},{0},\mathbb{F}]$.
%This even shows that there is no function $f(t)$ growing to infinity
%when $t$ tends to infinity such that $|B\cdot B|/|B|=f(|B|/|G|)$
%where $G$ is the period of the brick $B$. 
In the opposite direction one can show:
\begin{thm}\label{th13}
There exist an absolute constant $c>0$ and a constant $\alpha=\alpha(p)$ such that for any brick $B$ 
there exists a subgroup $G$ of $H_n$ satisfying 
$$
B\cdot B\cdot G=B\cdot B
$$ 
and
$$
\frac{|B\cdot B|}{|B|}\ge c\left(\frac{|B|}{|G|}\right)^{\alpha}.
$$
\end{thm}

This lower bound gives either a growing property for 
$B\cdot B$ in the case where $G$ is small or a rather regular 
structure for  $B\cdot B$ if $G$ is big.

Finally  we will show that the exponent $3/4+\varepsilon$ in Theorem \ref{th1} cannot be essentially reduced to less than $1/2$:

\begin{prop}\label{p2}
For any $n$ and $p$ 
there exists a brick $B\subseteq H_n$ such that
$$
|B|\ge\frac{\sqrt{p}}{4(2n)^n}|H_n|^{1/2}
$$ 
and the only cosets contained in $B\cdot B$ are cosets of the trivial  subgroup of $H_n$.
\end{prop}

Choosing $p$ large relative to $n$ in this result implies
the desired  effect.

\section{\bf  A big period or a big squaring constant}

\begin{proof}[Proof of Theorem \ref{th13}]

if $B=[\underline{X},\underline{Y},Z]$ is a brick in $H_n$ then
$|B\cdot B|\ge |2\underline{X}||2\underline{Y}|$ since
for any $\underline{x}\in 2\underline{X}$ and any 
$\underline{y}\in 2\underline{Y}$ there clearly exists
$z\in \mathbb{F}$ such that $[\underline{x},\underline{y},z]\in B\cdot B$. Thus if $k$ is the number of components
$X_i$, $Y_j$ such that $1<|X_i|,|Y_j|\le p/\sqrt2$, then
$|B\cdot B|\ge (\sqrt2)^k|B|$ since for such component $X_i$, $Y_j$
we have $$
|2X_i|\ge\min(p,2|X_i|-1)\ge \min(p,3|X_i|/2)\ge\sqrt2|X_i|
$$ 
and similarly $|2Y_j|\ge\sqrt2|Y_j|$ by the
Cauchy--Davenport Theorem (cf. Lemma \ref{l2}). If there exist at least two 
components $X_i$ and $X_j$ 
with cardinality bigger than $p/\sqrt2$ then there are two elements $w_i,w_j$ in $\mathbb{F}$
having at least $p/2$ solutions to the equations $w_i=x_i+x'_i$, $x_i,x'_i\in X_i$, and $w_j=x_j+x'_j$, $x_j,x'_j\in X_j$. Let $X'_i=X_i\cap(w_i-X_i)$ and $X'_j=X_j\cap(w_j-X_j)$.
We now observe that 
$B\cdot B$ contains $[X_1,\dots,X_{i-1},X'_i,X_{i+1},\dots,X_{j-1},X'_j,X_{j+1},\dots,X_n,
\underline{Y}]^2$. Since $|X'_i|,|X'_j|>p/2$ we have by the additive analogue of Lemma \ref{l1} $y_iX'_i+y_jX'_j=\mathbb{F}$ for any $y_i,y_j\in\mathbb{F}^*$. It follows
 that $B\cdot B$
fully covers the union  $[2X_1\times\cdots\times\{w_i\}\times\cdots\times\{w_j\}\times \cdots\times 2X_n,2\underline{Y},\mathbb{F}]$ of cosets of the center $[\underline{0},\underline{0},\mathbb{F}]$ of $H_n$.
 Note also
that for the indices $h$ such that $|Y_h|>p/2$ (or 
$|X_h|>p/2$ for $h\ne i,j$) we have  $2Y_h=\mathbb{F}$ (respectively
$2X_h=\mathbb{F}$).

We get a similar conclusion in the same way if we assume 
$|Y_i|,|Y_j|>p/\sqrt2$ for some $i\ne j$. Thus
if we denote by $\ell$ the numbers of components 
$X_h$ and  $Y_h$ with cardinality bigger than $p/\sqrt2$ and if we assume
$\ell\ge3$, then
$B\cdot B$ contains at least $(\sqrt2)^{k}$ cosets 
of a \textit{big} period $G$ of cardinality $p^{\ell-1}$, namely 
$$
G=[K_1,\dots,K_n,L_1,\dots,L_n,\mathbb{F}]
$$ 
where $K_i, L_j= \{0\}$ or $\mathbb{F}$.
By the facts that $|B|\le p^{k+\ell+1}$ and $\ell=\ln|G|/\ln p+2$, we get
$$
\frac{|B\cdot B|}{|B|}\ge\left(\sqrt2\right)^{k}\ge
\frac1{4} \left(\frac{|B|}{|G|}\right)^{\frac{\ln3}{2\ln p}}.
$$ 
If $\ell\le 2$ this bound still holds with $G=\{0\}$.
\end{proof}

\section{\bf Fourier analysis for a sum-product estimate}

We will use the following sum-product estimate:

\begin{prop}\label{p21}
Let $n,m\in \N,$ $X_1,X_2,\dots, X_n,Y_1,Y_2,\dots Y_n\subseteq \mathbb{F}^*=\mathbb{F}\setminus\{0\}$, $Z \subseteq \mathbb{F}$.
We have
$$
mZ+\sum_{j=1}^nX_j\cdot Y_j:=\Big\{z_1+\cdots+z_m+\sum_{j=1}^nx_jy_j,\ z_i\in Z, \ x_j\in X_j, \ y_j\in Y_j\Big\}=\mathbb{F},
$$
provided
\begin{equation}\label{card}
|Z|^2
\prod_{i=1}^n|X_i|
|Y_i|>p^{n+2}.
\end{equation}
\end{prop}

\begin{proof}
Let $X_i(t)$ (resp. $Y_i(t)$ and $Z(t)$) be the indicator of the set $X_i$ (resp $Y_i$ and $Z$).
One defines
$$
f_i(t)={1\over |X_i|}\sum_{a\in X_i}Y_i\big({t\over a}\big), \quad i=1,2,\dots ,n.
$$
Notice that $0\le f_i(t)\leq 1,$ and $f_i(t)>0$ if and only if $t\in X_i\cdot Y_i$. The Fourier transform of $f_i$ is
$$\widehat{f_i}(r)=\sum_xf_i(x)e(xr)
$$ 
where $e(x)=\exp(2\pi ix/p)$ as usual.

An easy calculation shows that for every $i=1,2,\dots ,n$
$$
\widehat{f_i}(r)={1\over |X_i|}\sum_{a\in X_i}\widehat{Y_i}(ra)
$$
and
\begin{equation}\label{eqn1}
\widehat{f_i}(0)={1\over |X_i|}\sum_{a\in X_i}\widehat{Y_i}(0)=|Y_i|,
\end{equation}
since $\widehat{Y_i}(0)=\sum_xY_i(x)=|Y_i|$.
Using the Cauchy inequality and the Parseval equality we get
if $p\nmid r$
\begin{equation}\label{eqn2}
|\widehat{f_i}(r)|\leq  {1\over \sqrt{|X_i|}}\sqrt{\sum_x|\widehat{Y_i}(x)|^2}=\sqrt{{p|Y_i|\over |X_i|}} 
\end{equation}

Let $u\in \mathbb{F}$. Let $S$ be the number of solutions of the equation
$$
u=z_1+z_2+\dots + z_m+ \sum_{j=1}^nx_jy_j, \quad z_i\in Z, \ x_j\in X_j,\  y_j\in Y_j.
$$
We can express $S$ by the mean of the Fourier transforms
of $Z$ and $f_i$ as follows:
$$
pS=\sum_{r\in \f}\widehat{Z}(r)^m\prod_{i=1}^n\widehat{f_i}(r)e(-ru).
$$
Our task is to show that this exponential sum is positive if the desired bound for the cardinalities
\eqref{card} holds.
Separating $r=0$ and using \eqref{eqn1} we can bound $S$ as
\begin{align*}
pS\geq& |Z|^m\prod_{i=1}^n|Y_i|-\sum_{r\neq 0}|\widehat{Z}(r)|^m\prod_{i=1}^n|\widehat{f_i}(r)|\\
\geq& |Z|^m\prod_{i=1}^n|Y_i|-|Z|^{m-2}\prod_{i=1}^n\sqrt{{p|Y_i|\over |X_i|}}\sum_{r\neq 0}|\widehat{Z}(r)|^2\\
\geq& |Z|^m\prod_{i=1}^n|Y_i|-p|Z|^{m-1}\prod_{i=1}^n\sqrt{{p|Y_i|\over |X_i|}}
%\geq& |Z|^m\prod_{i=1}^n|Y_i|-\max_{r\neq 0}\left\{|\widehat{Z}(r)|^{m-{2\over %n+1}}\prod_{i=1}^n|\widehat{f_i}(r)|^{1-\frac{2}{n+1}}\right\}\sum_{r\in \f}|%\widehat{Z}(r)|^{2\over n+1}\prod_{i=1}^n|\widehat{f_i}(r)|^{2\over n+1}
\end{align*}
by the Parseval equality and \eqref{eqn2}.
Hence $S>0$ whenever
$$
|Z|^2\prod_{i=1}^n|X_i||Y_i|>p^{n+2}.
$$
This completes the proof.
\end{proof}

\section{\bf Proofs of Theorem \ref{th1} and Proposition \ref{p2}}

\begin{proof}[Proof of Theorem \ref{th1}]

By the remark preceding Theorem \ref{th1} we may plainly 
assume that $|Z|<p/2$.

By the assumption on the brick $B$ we have
\begin{equation}\label{eqn3}
|B|=|Z|\left(\prod_{i=1}^n|X_i||Y_i|\right)>|H_n|^{3/4+\varepsilon}=p^{3n/2+3/4+\varepsilon(2n+1)}.
\end{equation}

For each $i$, there exists an element $a_i\in\mathbb{F}$
such that the number of solutions to the equation
$a_i=x_i+x'_i$, $x_i,x'_i\in X_i$, is at least $|X_i|^2/p$.
We denote by $\tilde{X_i}=X_i\cap(a_i-X_i)$ the set of the elements
$x_i\in X_i$ such that $a_i-x_i\in X_i$. We thus have $|\tilde{X_i}|\ge
|X_i|^2/p$. We similarly define $\tilde{Y_i}=Y_i\cap(b_i-Y_i)$
for some appropriate $b_i$ and also have 
$|\tilde{Y_i}|\ge|Y_i|^2/p$. It follows by  \eqref{eqn3} that
$$
|Z|^2\left(\prod_{i=1}^n|\tilde{X_i}||\tilde{Y_i}|\right)
\ge\frac{\left(|Z|\prod_{i=1}^n|{X_i}||{Y_i}|\right)^2}{p^{2n}}
>p^{n+3/2+\epsilon (4n+2)}.
$$
Hence for $n>1/8\epsilon$ we obtain from Proposition \ref{p21} that $2Z+\sum_{i=1}^n\tilde{X_i}\cdot \tilde{Y_i}=\mathbb{F}$ and
consequently
$$
B\cdot B\supseteq [(a_1,a_2,\dots ,a_n),(b_1,b_2,\dots ,b_n),\mathbb{F}],
$$
that is $B\cdot B$ contains at least one coset of  the non trivial subgroup
$G=[\underline{0},\underline{0},\mathbb{F}]$ of $H_n$.

In fact we may derive from the preceding argument a little bit more: for any index $i$ we have
$$
\sum_{a_i\in\mathbb{F}}|X_i\cap(a_i-X_i)|=|X_i|^2, \quad
\sum_{b_i\in\mathbb{F}}|Y_i\cap(b_i-Y_i)|=|Y_i|^2,
$$
hence
$$
\prod_{i=1}^n\left(\sum_{a_i\in\mathbb{F}}|X_i\cap(a_i-X_i)|\right)
\left(\sum_{b_i\in\mathbb{F}}|Y_i\cap(b_i-Y_i)|\right)
=\prod_{i=1}^n|X_i|^2|Y_i|^2,
$$
or equivalently by developing the product
\begin{equation}\label{eqnn10}
\sum_{\underline{a},\underline{b}\in\mathbb{F}^n}
\prod_{i=1}^n|X_i\cap(a_i-X_i)||Y_i\cap(b_i-Y_i)|
=\prod_{i=1}^n|X_i|^2|Y_i|^2.
\end{equation}
We denote by $E$ the set of all pairs $(\underline{a},\underline{b})\in\mathbb{F}^n\times \mathbb{F}^n$ such that
$$
|Z|^2\prod_{i=1}^n|X_i\cap(a_i-X_i)||Y_i\cap(b_i-Y_i)|>p^{n+2}.
$$ 
For such a pair 
$(\underline{a},\underline{b})$, the coset
$[\underline{a},\underline{b},\mathbb{F}]$ is contained
in $B\cdot B$ by the above argument. Then by  \eqref{eqnn10} 
$$
\left(\prod_{i=1}^n|X_i||Y_i|\right)|E|+
p^{n+2}(p^{2n}-|E|)>
\left(\prod_{i=1}^n|X_i||Y_i|\right)^2
$$
hence
$$
|E|>\frac{\prod_{i=1}^n|X_i|^2|Y_i|^2-p^{3n+2}}{\prod_{i=1}^n|X_i||Y_i|-p^{n+2}}.
$$
For $n>1/\epsilon$, we have by \eqref{eqn3} and the fact that $|Z|\le p$
$$
\prod_{i=1}^n|X_i||Y_i|>p^{3n/2+7/4},
$$
hence
$$
|E|\ge (1-p^{-3/2})\prod_{i=1}^n|X_i||Y_i|= (1-p^{-3/2})\frac{|B|}{|Z|}.
$$
Since $|Z|\le p/2$,
we thus have shown that $B\cdot B$ contains at least
$2(1-p^{-3/2})|B|/p\ge |B|/p$ cosets 
$
[\underline{a},\underline{b},\mathbb{F}]=[\underline{a},\underline{b},0] [\underline{0},\underline{0},\mathbb{F}], 
$
as we wanted.
\end{proof}

\begin{proof}[Proof of Proposition \ref{p2}]

Since $B$ is a brick, $B\cdot B$ is contained in a brick 
which takes the form $[\underline{U},\underline{V},W]$ where $\underline{U}$, $\underline{V}\subset\mathbb{F}^n$
are direct products of subsets of  $\mathbb{F}$ and $W\subset \mathbb{F}$. Since any non trivial subgroup of $H_n$ has
at least one of his $(2n+1)$ coordinate projections  equals to $\mathbb{F}$, it suffices to prove that neither 
$W$ is equal to $\mathbb{F}$, nor
$U$, nor $V$ contains a subset of the type
$\{x_1\}\times\cdots\times\mathbb{F}\times\cdots\times \{x_n\}$.

Let $B=[R,R,Z]$ where 
$$
R=\Big\{(r_1,r_2,\dots ,r_n)\in\mathbb{F}^n\,\mid\, 0\leq r_i < \sqrt{(p-1)/{2n}}\Big\}\quad
\text{and}\quad 
Z=\Big\{z\in\mathbb{F}\,\mid\, 0\leq z< {{p}/ 4}\Big\}.
$$ 
We have $|B|\ge p^{n+1}/4(2n)^{n}$ and $B\cdot B\subseteq[R+R,R+R,Z+Z+\langle R,R\rangle].$
Clearly $R+R\subseteq \Big[0,\sqrt{{2p}/n}\Big]^n,$ $Z+Z \subseteq \big[0,{(p-1)/2}\big)$ and $\langle R,R\rangle\subseteq \big[0,{(p-1)/2}\big]$.
Hence the statement.
\end{proof}

   \noindent{\bf Acknowledgement:} We thank the referee who gave a simplification in our proofs.

This work is supported by ``Balaton Program
   Project" T\'ET-09-1-2010-0056, and OTKA grants K-81658, K-100291.

\end{document}